\documentclass[11pt]{amsart}

\usepackage{amsmath, amssymb, amsthm, amsfonts, bbm, dsfont}
\usepackage{graphicx}

\usepackage{hyperref}

\usepackage{a4wide}

\numberwithin{equation}{section}
\theoremstyle{plain}
\newtheorem{theorem}[subsubsection]{Theorem}
 \newtheorem{lemma}[subsubsection]{Lemma}
 \newtheorem{prop}[subsubsection]{Proposition}
 \newtheorem{cor}[subsubsection]{Corollary}
 \newtheorem{conj}[subsubsection]{Conjecture}

 \theoremstyle{definition}

\newtheorem{remark}[subsubsection]{Remark}

\usepackage{hyperref}
\usepackage{pictexwd}
\usepackage{stackengine}
\usepackage{mathtools}
\usepackage{amssymb}
\usepackage{dsfont}

\usepackage{mathabx}

\newcommand{\bss}{\mathbb{S}}
\newcommand{\bii}{\mathbb{I}}

\newcommand{\CC}{\mathbb{C}}

\newcommand{\PP}{\mathbb{P}}

\newcommand{\ZZ}{\mathbb{Z}}
\newcommand{\bqq}{\mathbb{Q}}

\newcommand{\calF}{\mathcal{F}}

\newcommand{\calO}{\mathcal{O}}

\newcommand{\rzz}{\mathrm{Z}}

\newcommand{\rqq}{\mathrm{Q}}
\newcommand{\rtt}{\mathrm{T}}

\newcommand{\Res}{\textup{Res}}

\newcommand{\supp}{\mathrm{supp}}

\newcommand\rmT{\mathrm{T}}

\newcommand{\quash}[1]{}

\newcommand{\pt}{\textup{pt}}

\newcommand{\bare}{\mathrm{bare}}

\usepackage{tikz}
\usetikzlibrary{shapes,fit,intersections}
\usetikzlibrary{decorations.markings}
\usetikzlibrary{decorations.pathreplacing}
\usetikzlibrary{patterns}
\usetikzlibrary{matrix,arrows}
\usepgflibrary{decorations.pathmorphing}

\usepackage{tikz-cd}

\newcommand{\Hilb}{\textup{Hilb}}

\newcommand{\calL}{\mathcal{L}}

\newcommand{\calI}{\mathcal{I}}
\newcommand{\calE}{\mathcal{E}}

\newcommand{\ti}{\times}
\newcommand{\ot}{\otimes}

\newcommand{\cnt}{\mathbf{c}}
\newcommand{\ch}{\mathrm{ch}}
\newcommand{\cc}{\mathrm{c}}
\newcommand{\ee}{\mathrm{e}}
\newcommand{\PPP}{\mathrm{P}}
\newcommand{\PT}{\mathrm{PT}}
\newcommand{\DT}{\mathrm{DT}}

\title{EGL formula for DT/PT theory of local curves}
\author{A. Oblomkov}
\address{
A.~Oblomkov\\
Department of Mathematics and Statistics\\
University of Massachusetts at Amherst\\
Lederle Graduate Research Tower\\
710 N. Pleasant Street\\
Amherst, MA 01003 USA
}
\email{oblomkov@math.umass.edu}

\begin{document}
\begin{abstract}
  In this note we prove an integral formula for the bare one-leg PT vertex
  with descendents.
  The formula follows from the PT version of Ellingsrud-G\"ottsche-Lehn formula that
  is explained here. We apply the integral formula to obtain an elementary
  proof of rationality of one-leg capped PT vertex with descendents.
  We also obtain an
  integral formula for degree zero DT invariants with descendents. Finally
  we propose an explicit non-equivariant DT/PT correspondence as well as
  one descendent insertion fully-equivariant DT/PT formula.
\end{abstract}

\maketitle

\section{Introduction}
\label{sec:introduction}

Currently, there are many methods for enumerating sheaves on a three-fold \(X\) \cite{PandharipandeThomas15}, probably the most popular method is
known under the name PT theory  \cite{PandharipandeThomas15}. In this paper we explore a version of these invariants with the descendent insertions:
\[\langle\prod_i\ch_{k_i}(\gamma_i)\rangle_{X,\beta,\chi}^{\PT}=
  \int_{[\PPP(X,\beta)_\chi]^{vir}}\prod_i \ch_{k_i}(\gamma_i),\quad \gamma_i\in H^*(X),\]
the precise definition of \(\ch_k(\gamma)\) could be found in \cite{PandharipandeThomas09} or later in the current paper. The generating function of these invariants
\[\rzz^X_{\PT,\beta}(\prod_i \ch_{k_i}(\gamma_i)):=\sum_\chi \langle\prod_i \ch_{k_i}(\gamma_i)\rangle_{X,\beta,\chi} q^\chi,\]
is related to the analogous invariant in GW theory \cite{PandharipandePixton14}, \cite{OblomkovOkounkovPandharipande18}.

If \(\rtt=\CC^*\times\CC^*\times\CC^*\) acts on \(X\) and \(D\subset X\) is an equivariant divisor
then the definition extends to the equivariant relative, we denote the generating
function of corresponding invariant by: 
\[\rzz_{\PT,\beta}^{X/D}(\prod_i \ch_{k_i}(\gamma_i)|\omega)^{\rtt}\in \CC[q^{-1}][[q]]\ot\CC(t),\quad
  \CC[t_1,t_2,t_3]=H^*_{\mathrm{T}}(\pt),\]
for \(\omega\in H^*_\rtt(\Hilb_n(D))\).


The first goal of this paper is to write an analog of the celebrated EGL formula \cite{EllingsrudGoettscheLehn99} for the integral of the
tautological classes over the Hilbert scheme of points on the surface

\begin{theorem}\label{thm:PT-int}
  Let \(X=\calO(d_1)\oplus\calO(d_2)\)
  the generating function of the PT invariants 
  is given by
  \begin{equation}\label{eq:PT-int}
    \rzz^{X}_{\PT,n[\mathbb{P}^1]}(\prod_i \ch_{m_i}([0])\prod_j\ch_{l_j}([\infty]))^{\rmT}=\int_{|z_i|=R_i}\mathrm{F}^{d}_{\vec{m},\vec{l}}(q,t,z_1,\dots,z_{2n}).
  \end{equation}
  where 
  \(\mathrm{F}^d_{\vec{m},\vec{l}}\) is an explicit universal differential \(2n\)-form  and \(0\ll R_i\ll R_{i+1}\).
\end{theorem}

The exact formula for \(\mathrm{F}^{d}_{\vec{m},\vec{l}}\) is given in the main body of the note. The function \(\mathrm{F}^{d}_{\vec{m},\vec{l}}\)
is a version of the hypergeometric series and thus
has an integral presentation, itself. Thus we get a presentation of the PT invariant in the
question and contour integral some explicit meromorphic differential on \(\CC^{3n^2}\). 
We  connect this
formula with the Bethe anzatz results for \(J\)-functions of quasi-maps in a forthcoming publication \cite{Oblomkov19}.


Another consequence of the formula is a more elementary proof of the rationality of the one-leg PT vertex with the descendents \cite{PandharipandePixton12}, \cite{Smirnov16a}:

\begin{cor} \cite{PandharipandePixton12}, \cite{Smirnov16a}
  Let \(X=\CC^2\ti \PP^1\), \(\omega\in H^*(\Hilb_n(\CC^2))\) then
  \[\rzz^{X/\CC^2\ti \infty}_{\PT,n[\PP^1]}(\prod_i\ch_{k_i}([0])|\omega)^{\rtt}\in \CC(q)\otimes \CC(t).\]
\end{cor}

  To be more precise we provide a simplified version of the key Proposition 4 of \cite{PandharipandePixton12} that states that the
  bare vertex is a rational function of \(q\) under specialization \((t_1+t_2)=ct_3\), \(c\in \ZZ_+\), see section~\ref{sec:spec-equiv-param}
  for more details.



  
The DT theory precedes PT theory by a few years \cite{MaulikNekrasovOkounkovPandharipande06a}, similarly to already discussed PT theory
one defines the generating series of DT invariants with descendent insertions: \[\rzz^X_{\DT,\beta}(\prod_{i=1}^n\ch_{k_i}(\gamma_i))=
  \sum_\chi q^\chi\int_{[\Hilb(X,\beta)_\chi]^{vir}}\prod_i \ch_{k_i}(\gamma_i),\quad \gamma_i\in H^*(X).\]

The rationality statement is more complicated for the DT theory, arguably because the degree zero DT invariants
\(\rzz^X_{\DT,0}(\dots)^{\rmT}\) are non-zero, in contract to PT theory.  In particular, it is conjectured by Okounkov-Pandharipande that the degree zero invariants
\[\rzz^X_{\DT,0}(\prod_{i=1}^n\ch_{k_i}(\gamma_i))^{\rmT}/\rzz^X_{\DT,0}(1)^{\rmT}\] are  \(H^*_{\rmT}(\pt)\)-linear combinations of the monomials of the derivatives of the functions:
\[\mathfrak{F}_{k}:=\sum_{n>0}n^{2k+1}\frac{q^n}{(1-q^n)}.\]

We propose a minor modification of the Okounkov-Pandharipande conjecture, it was probably known to them. We expect the degree zero invariants are the only source of irrationality of DT invariants with descendents. 

\begin{conj}\label{conj:DT0}
  For any \(k_1,\dots,k_n\) and \(\gamma_i\in H^*_{\rmT}(X)\) we have
  \[\rzz_{\DT,\beta}^{X}(\prod_i\ch_{k_i}(\gamma_i))^\rtt=\sum_{\vec{l}\le \vec{k}} \mathrm{C}^{\vec{l}}_{\vec{k},\beta}(\vec{\gamma})\cdot \rzz_{\DT,0}^X(\prod_i\ch_{l_i}(\gamma_i))^\rtt,\]
  where \(mathrm{C}^{\vec{l}}_{\vec{k}}\)  is a rational function of \(q\) with coefficients in \(H^*_\rtt(X)\).
  Moreover there is a universal expression for \(\mathrm{C}^{\vec{l}}_{\vec{k},\beta}(\vec{\gamma})\) in terms of \(c_i(X)\), \(\rzz^X_{\PT,\beta}(\prod_i\ch_{m_i}(\gamma_i))^\rtt\),
  \(\vec{m}\le \vec{l}\).
\end{conj}

The statement is known to be true for toric varieties in the case of toric varieties \cite{MaulikOblomkovOkounkovPandharipande11} under assumptions
\(k_i\le 2\)\footnote{\(\ch_0(\gamma)=\int_X\gamma\) and \(\ch_1(\gamma)=0\).}. In this case PT invariants are essentially equal to the DT invariants:
\begin{equation}\label{eq:simple}
  \rzz_{\DT,\beta}^X(\prod_i\ch_{k_i}(\gamma_i))^{\rmT}=\rzz_{\PT,\beta}^X(\prod_i\ch_{k_i}(\gamma_i))^{\rmT}\rzz_{ \DT,0}^X(1)^{\rmT}.
\end{equation}

The last equation also holds at non-equivariant limit if \(\gamma_i\in H^{\ge 2}(X).\) In the last section of this paper we
present a conjectural formula extending \eqref{eq:simple} to the non-equivariant non-stationary case. 

In general, this paper
was motivated by the hope that EGL-type formalism for DT invariants  would provide a path to precise DT/PT formula
in the fully equivariant case. Unfortunately, we fell short of this goal but we found some interesting relations between
the DT and PT invariants in the case of the theories of local curves.

The  most natural form of the formula is in terms of the modified PT tautological classes \(\ch'_k\). The kernels of
map \([\calO_x\to \calF]\)  in two-step complex representing stable pair glue into a universal sheaf \(\bqq\) over \(P(X,\beta)_\chi\)
and respectively we define
\[\ch'_k=\pi_*^X\ch_k(\bqq).\]

The \(\CC^*\ti\CC^*\ti \CC^*\)- fixed points on \(\Hilb_n(\CC^3)\) are naturally labeled by the plane partitions of size \(n\),
which are also known under the name 3D Young diagrams.
By slicing a 3D Young diagram \(\pi\) the planes parallel to \(x,y\)-plane we represent it as a sequence of nested 2D Young diagrams:
\(\pi^{(1)}\supset \pi^{(2)}\supset\dots\). We denote by
\[\rzz^{\CC^3}_{\DT,0,\mu}(\prod_i \ch_{k_i}(\gamma_i))^T\]
the subsum of the localization formula for \(\rzz^{\CC^3}_{\DT,0}(\prod_i \ch_{k_i}(\gamma_i))^\rtt\) that involves only the torus
fixed point with the labels \(\pi\) such that \(\pi^{(1)}=\mu\).

The PT side of our formula is given by the bare one-leg vertex with descendents:
\[\rzz^{\bare}_{\PT}(\prod_i \ch_{k_i}(\gamma_i)|\mu)^\rtt,\]
which has a localization formula similar to the formula for \(\rzz^{\CC^3}_{\DT,0,\mu}(\prod_i \ch_{k_i}(\gamma_i))^\rtt\), see section \ref{sec:one-leg-invariants} for the precise definition. The main formula in theorem below gives a precise meaning to
the mentioned similarity:

\begin{theorem}\label{thm:DTPTdeg0}
  For any \(m,n\) and \(\mu,|\mu|=n\) there is an explicit meromorphic function \(f(\vec{k},\vec{z},\vec{w})_\mu\), \(\vec{k}\in \CC^n,\vec{w}\in\CC^m\) such that for a generic values of the equivariant parameters we have
  \[\int d\vec{z} \sum_{\vec{k}\ge 0}f(\vec{k},\vec{z},\vec{w})_\mu=(-1)^m\rzz^{\bare}_{\PT}\bigg(\frac{\ee(\bqq)^2
      \ee(\bqq\ot (\calO^{[n]})^\vee\ot \Lambda^\bullet T_S)
     }{
       \ee(\bqq\ot K_X)^2\ee(\bqq\ot (\calO^{[n]})^\vee\ot K_X\ot\Lambda^\bullet T_S)
       } \prod_{i=1}^m\ch'[w_i]\bigg|\mu\bigg)^{\rtt},\]
  \[\int d\vec{z} \sum_{\vec{k}\ge -1}f(\vec{k},\vec{z},\vec{w})_\mu=\rzz^{\CC^3}_{\DT,0,\mu}\bigg(\prod_i \ch[w_i]\bigg|\mu\bigg)^{\rtt}\]
  where \(X=\CC^2\ti \PP^1\), \(S=\CC^2\), \(\ch'[w]\), \(\ch[w]\) are the generating functions for the descendents \(\ch'_k\) and \(\ch_k(1)\) and \(\ee(\bqq\ot \calL)=
\ee(\pi^X_*(\bqq\ot \calL))\) is the Euler class of the corresponding sheaf.
  The integration should be interpreted as an iterated residue at \(z_i=\infty\) evaluation.
\end{theorem}

The first formula in the theorem is essentially an integral presentation of the bare vertex \(\rzz^{\bare}_{\PT}(\dots)\)
from the section \ref{sec:one-leg-invariants}. Unfortunately, we do not have a geometric interpretation
for the ratio of  Euler classes  in the formula.
However, since \(\rzz_{\DT,0}^{\bare}(\dots)\) is the sum of  \(\rzz_{\DT,0,\mu}^{\CC^3}(\dots)\) over all possible \(\mu\), one can
use the previous theorem to write an iterated residue formula for the degree zero \(\DT\) invariants.

The main body of the note is divided on four sections. In the section \ref{sec:egl-form-integr} we rewrite the classic EGL
formula as an iterated residue formula. In the next section \ref{sec:local-as-cont} we recall the localization formalism
for PT invariants with descendents and prove an iterated residue formula for the bare PT vertex with descendents, this
formula implies theorem~\ref{thm:PT-int} almost immediately. Section \ref{sec:local-as-cont} provides an analysis for the
ratio of the PT and DT localization measures and a derivation of the residue formula for degree \(0\) DT invariants from
theorem \ref{thm:DTPTdeg0}. Finally, in the section \ref{sec:dtpt-corr-expl} we review known rationality conjectures and
theorem for DT/PT invariants with descendents and state new DT/PT conjectures in non-equivariant setting (see section \ref{conj:Frank})
and one descendent insertion DT/PT conjecture in the fully-equivariant setting (see section \ref{sec:equiv-wall-cross}).

{\bf Acknowledgements:} I would like to thank Andrei Okounkov and Rahul Pandharipande for introducing me to
world of enumerative algebraic geometry and for many enlightening mathematical conversations.
The results in this paper  would not be found without  our joint
projects with them. I also would like to thank James Hagborg for his diligent work during 2017 REU program,
the equivariant DT/PT formula presented in the last section  is crowning jewel of the program.
I am  thankful to Andrei Negut and Tom Bridgeland for many useful conversations related to results of the paper.
I also would like to thank FIM-Institute for Mathematical Research at ETH, Zurich for excellent working
conditions during my two-week visit during Fall of 2018.
This paper is based upon work supported by the National Science Foundation under Grant No. 1440140, while the author was in residence at the Mathematical Sciences Research Institute in Berkeley, California during the Spring semester of 2018.
The author was partially supported by NSF CAREER grant
DMS-1352398 and NSF grant DMS-1760373. 


\section{EGL formulas: integral presentation}
\label{sec:egl-form-integr}

The key geometric component of our computation is inductive formula for computation of the tautological
 integrals over \(\Hilb_n(S)\) from \cite{EllingsrudGoettscheLehn99}.
We recall the details of the construction below.

\subsection{Nested Hilbert schemes}
\label{sec:nest-hilb-schem}

In this section $S=\CC^2$ but most of statements are valid for any
$S$. Two-dimensional torus $T=(\CC^*)^2$ acts on $S$. I work
$T$-equivariantly. The tangent weights at the origin of $\CC^2$
are $t_1,t_2$.

Central object in the computation is the one step nested Hilbert
scheme $\Hilb_{n+1,n}(S)$ of points on the surface $S$. I follow the
notations from \cite{EllingsrudGoettscheLehn99}.  The scheme  $\Hilb_{n+1,n}(S)$  consists of
the pairs of ideals $I\in \Hilb_{n+1}(S)$, $J\in \Hilb_{n}(S)$ such that
$I\subset J$. It is smooth scheme. If $U_n\subset \Hilb_n(S)\times S$
is a universal subscheme of $\Hilb_n(S)$:
$$ (I,z)\in U_n \mbox{ iff } z\in \supp(I)$$
then $\Hilb_{n+1,n}(S)$ is blow up of $\Hilb_n(S)\times S$ along $U_n$.
Notice $U_n$ itself is  singular (and is not lci) for $n>2$.

The following map are important for our constructions:
$$\Psi:\Hilb_{n+1,n}( S)\to \Hilb_{n+1}(S)$$ is
$n$-fold finite cover. The map
$$\Phi: \Hilb_{n+1,n}(S)\to \Hilb_n(S)$$ has $2$ dimensional fibers.
The support of $J/I$ is a point and the map:
$$\rho: \Hilb_{n+1,n}(S)\to S,\quad (I,J)\mapsto \supp(J/I)$$ is well defined.
The maps $\rho$ and $\Phi$ compose the birational map:
$$\Sigma=(\Phi,\rho): \Hilb_{n+1,n}(S)\to \Hilb_n(S)\times S.$$

\subsection{Sheaves and tautological classes}

The universal ideal sheaf $\mathcal{I}_n$ over $\Hilb_n(S)\times S$ fits into the
exact sequence:
$$0\to \mathcal{I}_n\to \mathcal{O}\to \mathcal{O}_n\to 0.$$
We obtain rank $n$ vector bundle on $\mathcal{O}^{[n]}$ by pushing
forward the sheaf $\mathcal{O}_n$
on $\Hilb_n(S)$. That is
$$\pi_{*}(\mathcal{O}_n)=\mathcal{O}^{[n]},$$
where $\pi$ is a projection on the first factor of $\Hilb_n(S)\times S$.

The following line bundle:
$$ \mathcal{L}:=\Psi^*(\mathcal{O}^{[n+1]})/\Phi^*(\mathcal{O}^{[n]}) $$ is central for the computations below.
In \cite{EllingsrudGoettscheLehn99} it is described as a tautological line bundle.

Let us denote by $\sigma_i$ the integral of charcteristic class
$\sigma_i=\pi_*(\ch_{i+2}(\mathcal{O}_n))\in H^*_T(\Hilb_n(S))$. The equivariant GRR implies
\begin{equation}\label{eq:ChOn}
\sigma=\sum \sigma_i=\ch(O^{[n]})\pi_*(1-\ch(T_S)+\ch(K_S^\vee))=
\ch(O^{[n]})(1-e^{t_1})(1-e^{t_2})/(t_1t_2).
\end{equation}

It is known that $\sigma_i$ generate cohomology ring of $\Hilb_n(S)$.


\subsection{Relations between sheaves and characteristic classes}

In this section we describe the geometric input in our construction.

\begin{prop}\cite{EllingsrudGoettscheLehn99} For  $k>0$ we have
\begin{equation}\label{eq:pushf_c1}\Sigma_*(\cc_1(\mathcal{L})^k)=(-1)^k \cc_k(\mathcal{O}_n).\end{equation}
\end{prop}

I prefer to work with the maps $\Phi$ and $\Psi$ and to avoid
$\Sigma$, so we reformulate the statement from \cite{EllingsrudGoettscheLehn99}.
For  $k>0$ we have
\begin{equation*}\Phi_*(\cc_1(\mathcal{L})^k)=(-1)^k \pi_*(\cc_k(\mathcal{O}_n)).\end{equation*}
Now let us notice that the relation \eqref{eq:ChOn} implies that  inside \(H^*(S^{[n]}\times S)\) we have relation:
\[\ch(\calO_n)=\pi^*(\ch(\calO^{[n]})-\ch(\calO^{[n]}\otimes t_1)-\ch(\calO^{[n]}\otimes t_2)+\ch(\calO^{[n]}\otimes t_1t_2)).\]
Thus RHS of the equation \eqref{eq:pushf_c1} could be rewritten with the help of:
\begin{equation}\label{eq:pi_cOn}
  \cc(\calO_n)=\pi^*\left(\frac{\cc(\calO^{[n])})\cc(\calO^{[n]}\otimes t_1t_2)}{\cc(\calO^{[n]}\otimes t_1)\cc(\calO^{[n]}\otimes t_2)}\right).\end{equation}

The last formula provides us with the method of reducing of the integral of the tautological classes \(\ch_k(\calO^{[n]})\) over
\(\Hilb_n(S)\) to the integral of some expression of \(\ch_i(\calO^{[n-1]})\) over \(\Hilb_{n-1}(S)\). Multiple iteration of this induction could be
encoded as contour integral. We use notations of \cite{Negut15} where this formula appeared. Let us define \[\cc(V)[u]=\prod_{i=1}^m(1-ur_i),\]
where \(r_i\) are Chern roots of the vector bundle \(V\). Let us also fix notation for
\[\omega(z):=\frac{z(z-t_1-t_2)}{(z-t_1)(z-t_2)}.\]
The following equation could be deduced from the results of \cite{Negut15} by applying RR systematically to translate the results of \cite{Negut15} from
K-theory to homology. However, we would like to show how one can derive this
formula from the result of \cite{EllingsrudGoettscheLehn99}.
\begin{prop}\label{prop:Hilb-int} For any \(m,n\ge 0\) we have:
  \[\int_{S^{[n]}}\prod_{i=1}^m \cc[u_i](\calO^{[n]})=\frac{1}{n!}\int \left(\prod_{1\le i<j\le n}\omega(z_i-z_j)\right)\prod_{k=1}^n\left( \prod_{l=1}^m(1-u_lz_k)\right)Dz\]
  where \(Dz:=\prod_{i=1}^n\frac{dz_i}{z_i}\) and the integral in LHS is taken over the contour \(|z_i|=R_i\) with \(R_1\gg R_2\gg\dots\gg R_n\gg 0\).
  
\end{prop}
\begin{proof}
  Our proof is inductive and essentially just a translation of the result of \cite{EllingsrudGoettscheLehn99} into an integral formula.
  We show:
  \begin{equation}\label{eq:intSn}
    \int_{S^{[n]}}\prod_{i=1}^m c[u_i](\calO^{[n]})=\int_{S^{[n-1]}}\int_{|z|=\epsilon}\prod_{i=1}^m(1-u_iz)\cc[u_i](\calO^{[n-1]})\cdot\prod_{i=1}^{n-1}\omega(r_i-z)Dz,
    \end{equation}
  where \(r_i\) are the Chern roots of the vector bundle \(\calO^{[n-1]}\).

  First let us observe that the map \(\Psi: \Hilb_{n,n-1}(S)\to \Hilb_n(S)\) is a \(n\)-branched  cover and thus \[\int_{S^{[n]}}X=\frac{1}n\int_{S^{[n,n-1]}}\Psi^*(X).\]
  On the other hand we can use localization description of the pull-back: if \(\pi:Z\to Y\) is map of smooth varieties and \(\pi\) sends
  an isolated torus-fixed point \(z\in Z\) to the isolated torus fixed point \(y\in Y\) then in localized cohomology we have
  \[\pi^*\left(\frac{[y]}{e(T_y)}\right)=\frac{[z]}{e(T_z)}.\]

  The torus-fixed points of \(\Hilb_n(S)\) are monomial ideals, naturally labeled by the partitions \(\lambda\) of size \(n\) and the torus-fixed points of
  \(\Hilb_{n,n-1}(S)\) are pairs of the monomial ideals labeled by the partitions \(\lambda\supset\mu\). Thus the LHS of \eqref{eq:intSn} can be rewritten as
  \begin{equation}\label{eq:psi}
  \sum_{|\lambda|=n}\frac{1}{e_\lambda}\prod_{k=1}^m\prod_{ij\in\lambda}(u_k-\cnt(ij))=\frac{1}{n}\sum_{\lambda,\mu}\frac{1}{e_{\lambda,\mu}}
    \prod_{k=1}^m\prod_{ij\in \lambda}(u_k-\cnt(ij)).\end{equation}
  where \(e_{\dots}\) are the Euler classes of the corresponding normal bundles and \(\cnt(ij)=\cnt_\lambda(ij)\) is the content of the square inside of the Young
  diagram:
  \[\cnt(ij)=it_1+jt_2.\]
  The product under the integral can be rewritten in slightly different form:
  \[(u-\cnt(l,m))\prod_{ij\in\mu}(u-\cnt(ij)),\]
  where \((l,m)=\lambda\setminus\mu\) and it is exactly the weight of the line bundle \(\mathcal{L}\) at the point \((\lambda,\mu)\in S^{[n,n-1]}\).
  Now let us recall the localization formula for the push-forward. If \(\pi:Z\to Y\) is map of smooth varieties and \(\pi(z)=y\), \(z,y\) are
  isolated torus-fixed points, then \(\pi_*([z])=[y]\). In these terms the main geometric input of \cite{EllingsrudGoettscheLehn99} can be restated
  as a combinatorial formula for the Chern class of \(\calO_{n-1}\):
  \[\frac{\mathrm{c}_k(\calO_{n-1})_\mu}{(t_1t_2)e_\mu}=\sum_{\lambda}\frac{\left(\cnt(\lambda\setminus\mu)\right)^k}{e_{\lambda,\mu}}\]
  with the sum over all \(\lambda\) such \(\mu\subset\lambda\).
  Since \[\int_{S^{[n,n-1]}}X=\int_{S^{[n-1]}}\Phi_*(X)\] we can rewrite the sum in RHS of (\ref{eq:psi}) as
  \[\int Dz\sum_\mu\frac{1}{e_\mu}\prod_{i=1}^m\left( \cc_{s_i}(\calO^{[n-1]})(1-u_iz)\right)\pi_*(\sum_k \mathrm{c}_{k}(\calO_{n-1})z^{-k}).\]
  Since the total Segre class is the inverse of the total Chern class, the formula (\ref{eq:pi_cOn}) explains the appearance of \(\omega\) in
  our inductive formula.
\end{proof}

Let us fix notation \(\Omega(z_1,\dots,z_n)\) for the for the integration
weight in the previous theorem:
\[\Omega(z_1,\dots,z_n)= \prod_{k=1}^n\frac{dz_i}{z_i}\prod_{1\le i<j\le n}\omega(z_i-z_j)\]
\begin{remark}\label{rem:ord}
  The iterated integral in the previous theorem has a clear algebraic interpretation as an iterated residue. That is the
  integral can replaced with \(\Res_{z_1=\infty}\Res_{z_2=\infty}\dots\Res_{z_n=\infty}\). The order of taking residues is important, switching
  order of taking residues changes the answer.
  \end{remark}

\section{Localization as contour integral}
\label{sec:local-as-cont}

\subsection{PT descendents: definitions}
\label{sec:pt-desc-defin}

The  moduli space of pairs \(P_n(X,\beta)\) parametrizes stable pairs \([\calO_X\to \calF]\) where \(\calF\) is
a pure one-dimensional sheaf \([\supp(\calF)]=\beta\) and the cokernel of the map has zero-dimesional support,
\(\chi([\calO_X\to \calF])=n\).

If \(\dim X=3\) then the  moduli space \(P_n(X,\beta)\) poses a virtual cycle \([P_n(X,\beta)]\) of degree \(\beta\cdot c_1\).
The  sheaves \(\calF\) glue into  a universal sheaf \(\mathbb{F}\)  over the product \(P_n(X,\beta)\). Thus we can define
the descendent classes from the introduction by
\[\ch_k(\alpha)=\pi_*^X(\ch_k(\mathbb{F})\cdot \alpha),\quad \alpha\in H^*(X),
\]
where \(\pi^X\) is the projection along \(X\).

\subsection{One leg invariants with descendents}
\label{sec:one-leg-invariants}

In this section we review the localization formula for the  bare one-legged PT vertex and we derive a hypergeometric presentation of
the localization that could be translated to a residue formula for the relative PT invariants.

Before we state our main formula let us
choose a basis in \(H^*_{T}(\Hilb_n(\CC^2))\). The localization formula implies that the monomials \[\cc_\lambda:=\prod_{i=1}^\ell \cc_{\lambda_i}(\calO^{[n]}),\quad
|\lambda|=\sum_{i}\lambda_i=n.\]
form a basis in the equivariant cohomology.
In the original paper \cite{MaulikNekrasovOkounkovPandharipande06a} a formula for the equivariant localization of DT invariants is  proven.
The answer is given in terms of bare vertex. The analogous formalism for PT theory was developed in \cite{PandharipandeThomas09}. For
exact geometric meaning of the bare vertex we refer to \cite{PandharipandeThomas09}. Here we just remind   the bare vertex
\[\mathrm{Z}^{\bare}_\PT(\prod_i\ch_i|\gamma), \quad \gamma\in H^*_T(\Hilb_n(\CC^2))\] is an element of the ring of formal power-series
\(\CC(t_1,t_2,t_3)[q^{-1}][[q]]\) such that we can evaluate the absolute PT invariants of the local \(X=\calO(d_1)\oplus\calO(d_2)\to\mathbb{P}^1\) by
\begin{equation}\label{eq:locP1}
  \langle \prod_i\ch_{m_i}([0])\prod_{j}\ch_{l_j}([\infty])\rangle_{n[\mathbb{P}^1]}^{\PT}=\sum_{s,r}\mathrm{Z}^{\bare}_{\PT}(\prod_i\ch_{m_i}|\gamma_s)
  \mathrm{E}^{d}(\gamma_s,\gamma_r)
\left.  \mathrm{Z}^{\bare}_{\PT}(\prod_j\ch_{l_j}|\gamma_r)\right|_{t_i=s_i}\end{equation}
where \(\gamma_i\) is a basis of \(H^*_{T}(\Hilb_n(\CC^2))\), \(s_3=-t_3,s_2=t_2+d_2t_3,s_1=t_1+d_1t_3\) and edge-term in the basis of
the torus fixed points is given by:
\[\mathrm{E}^{d_1,d_2}(\lambda,\mu)=\delta_\mu^\lambda\cdot \mathrm{Exp}( -\mathrm{E}^{d}(\lambda)), \quad \mathrm{E}^d(\lambda)=
t_3^{-1}\frac{\mathrm{F}_e(t_1,t_2)}{1-t_3^{-1}}-\frac{\mathrm{F}_e(t_1t_3^{-d_1},t_2t_3^{-d_2})}{1-t_3^{-1}},\]
with \(\mathrm{F}_e\) defined by \eqref{eq:Fe} and \(\mathrm{Exp}\) is a plethistic exponential   \eqref{eq:pleth}.

We also adopt the following definition of Pochhammer symbol:
\[ [x]_n=\frac{\Gamma(x+n)}{\Gamma(x)}.\]

\begin{theorem}\label{thm:mainPT} For any partition \(\lambda\) of size \(n\) we have an integral formula for the bare PT one-leg vertex
  \[\mathrm{Z}^{\bare}_{\PT}( \prod_{r=1}^m \ch[u_r]|c_\lambda)^{\PT}_{n[\mathbb{P}^1]}=\sum_{\vec{k}\in \ZZ_+^n} \mathrm{Z}^{\bare}_{\PT,\vec{k}}( \prod_{r=1}^m \ch[u_r]|c_\lambda)q^{|\vec{k}|}\]
  \[\mathrm{Z}^{\bare}_{\PT,\vec{k}}( \prod_{r=1}^m \ch[u_r]|c_\lambda)=
    \frac{1}{n!}\int  \Omega(z) e_\lambda(z)  \sum_{\vec{k}\ge 0}\Pi(\vec{k},z)
    \prod_{r=1}^m\sum_{i=1}^n \frac{e^{t_3u_r(z_i+k_i)}}{\delta(u_r)},\]
  \[\Pi(\vec{k},z)=\prod_{i=1}^n\frac{[-z_i-a_1-a_2]_{k_i}
    }{[-z_i+1]_{k_i}}\prod_{1\le i<j\le n}
    F^{-1}_{k_j-k_i}(z_i-z_j),\]
  \[F_k(z)=
    \frac{[z-a_1]_{k}[z-a_2]_{k}[z+a_1+a_2]_{k}}
    {[z+a_1]_{k}[z+a_2]_{k}[z-a_1-a_2]_k}\]
  where  \(\delta(z)^{-1}=(1-e^{zt_1})(1-e^{zt_2})\), \(a_i=t_i/t_3\) and the contour integration is the same as in the previous theorem.
\end{theorem}

\subsection{Combinatorics of the torus-fixed locus}
\label{sec:comb-torus-fixed}

Before we prove our theorem we remind the setting of the equivariant localization for the relative invariants. The torus fixed points of
the moduli space \(P_\chi(\mathbb{P}^1\times \CC^2,n[\mathbb{P}^1])\) are naturally labeled by the pairs of inverted 3D partition of shape
\(\lambda\), and with the constraint on the number of boxes. Let us describe the corresponding combinatorial object.

We try to stay as close as it possible to the notations of \cite{PandharipandeThomas09}. Let us denote by \(\mathrm{Cyl}_{\lambda}\) the
infinite cylinder going in \(t_3\) direction and having cross-section \(\lambda\). The points inside the cylinder correspond exponents of the elements
the monomial basis of the quotient \[M_\lambda:=\CC[x_1,x_2]/I_\lambda\otimes \CC[x_3,x_3^{-1}],\quad
\calO_{\lambda}:=M_\lambda\cap \CC[x_1,x_2,x_3].\]

The restricted partition \(\pi\) of the shape \(\lambda\)
is the set of the exponents of the monomial basis of a \(\CC[x_1,x_2,x_3]\)-submodule of the submodule \[\calO_\lambda\subset M_\pi\subset M_\lambda\] such that the {\it size}
\(\ell(\pi):=\dim M_\pi/\calO_\lambda\) is finite. Given a restricted partition \(\pi\) the complement \(\mathrm{Cyl}_\lambda\setminus \pi\)  is
an infinite 3D partition.
Let us denote the set of restricted 3D partitions of the shape \(\lambda\) as \(\Pi^{\PT}_\lambda\) and \(\Pi^{\PT}_{\lambda,\chi}\) is the set of the restricted partitions
of shape \(\lambda\) and of size \(\chi\).

We also use
notation \(\hat{\pi}\) for  the support
of \(M_\pi/\calO_\lambda\).  For  \(\pi\in\Pi^{\PT}_\lambda\) the generating function
for the localiztion weights of Chern characters of  \(\ch_m([0])\) at the
torus-fixed point \(\pi\) is:
\[\ch(\pi)[z]=(1-e^{t_1z})(1-e^{t_2z})(1-e^{t_3z})\sum_{(i,j,k)\in \pi}e^{(it_1+jt_2+kt_3)z}.\]

Following notations of \cite{PandharipandeThomas09} we denote by \(\mathrm{F}_v(\pi)\in\ZZ(t_1,t_2,t_3)\) the generating function of the
points inside the partition \(\pi\). The expression \(\mathrm{F}_v(\pi)\) is the rational function of the equivariant parameters \(t_1,t_2,t_3\).
As in \cite{PandharipandeThomas09} we define the  generating functions \[\mathrm{Q}_e(\lambda)=\sum_{(i,j)\in\lambda} t_1^it_2^j,\quad
\mathrm{Q}_v(\pi)=\mathrm{F}_v(\pi)-\mathrm{Q}_e(\lambda)/(1-t_3).\] The last sum is a generating functions for the monomials in the basis
of the finite dimensional quotient \(Q_\pi:=M_\pi/\calO_\lambda\).

\subsection{Localization weight}
\label{sec:localization-weight}

Let us fix notation \(f\mapsto \bar{f}\)  for the involution \(\CC(t_1,t_2,t_3)\) defined by \(t_i\mapsto t_i^{-1}\). In these terms \cite{PandharipandeThomas09}
define the generating function for the equivariant weight of the torus-fixed point \(\pi\) inside the PT moduli space:
\begin{equation}\label{eq:V_PT}
\mathrm{V}_v^{\PT}(\pi):=\mathrm{F}_v-\frac{\overline{\mathrm{F}}_v}{t_1t_2t_3}+\mathrm{F}_v\overline{\mathrm{F}}_v\frac{(1-t_1)(1-t_2)(1-t_3)}{t_1t_2t_3}+
  \frac{\mathrm{F}_e}{1-t_3},\end{equation}
\begin{equation}\label{eq:Fe}
  \mathrm{F}_e=-\mathrm{Q}_e-\frac{\mathrm{Q}_e}{t_1t_2}+\mathrm{Q}_e\overline{\mathrm{Q}}_e\frac{(1-t_1)(1-t_2)}{t_1t_2}
\end{equation}

The equivariant weight of the localization is obtain by taking plethistic exponent:
\begin{equation}\label{eq:pleth}
  \mathrm{Exp}(\sum a_{ijk}t_1^it_2^jt_3^k)=\prod_{ijk}(it_1+jt_2+kt_3)^{a_{ijk}}\end{equation}
Respectively, the virtual Euler class at the torus fixed point \(\pi\) is given by \[\calE_\pi^{\PT}:=\mathrm{Exp}(-\mathrm{V}^\PT_v(\pi)).\]

\subsection{Proof of theorem~\ref{thm:mainPT}}
\label{sec:proof-induct-form}

The theorem~\ref{thm:mainPT} follows from the inductive formula that allows us to reduce the size of the partition by one. The bare vertex is defined to be the sum
\[\mathrm{Z}^{\bare}_\PT( \prod_{r=1}^m \ch[u_r]|c_\lambda)=
\sum_{|\mu|=n}\frac{c_\lambda(\mu)}{e_\mu}\sum_{\pi\in\Pi_{\mu}^\PT}\calE_\pi^{\PT}\prod_{r}\ch[u_r](\pi).
\]

The main observation of the paper is that the second sum is equal to the sum over pairs
\(k\in \ZZ_{\ge 0}\) and \(\pi'\in\Pi_{\mu'}^\PT\), \(|\mu'|=n-1\), \(\mu'\subset\mu\).
The term corresponding to a pair \(k,\pi'\) is the product \(\calE_{\pi'}^{\PT}\) and
\begin{equation}\label{eq:ExpDiff}
  \frac{[-z_n-a_1-a_2]_{k}}{[-z_n+1]_{k}}\prod_{ij\in\lambda'} F_{k-k_{ij}}^{-1}(\cnt(ij)/t_3-z_n)
  \prod_{r=1}^m 
  \sum_{ij\in\lambda} \frac{e^{(\cnt(ij)+k_{ij}t_3)u_r}}{\delta(u_r)}
\end{equation}
where \(\pi'\in \Pi_{\lambda'}^\PT\), \(\pi'=\{(i,j,l),l\ge -k_{ij},ij\in\lambda'\}\) and
\(z_n=\cnt(i_0,j_0)\), \(k=k_{i_0j_0}\) \(\lambda\setminus\lambda'=(i_0,j_0)\). 

The formula follows from the analysis of the equivariant weights in the localization formula. Let \(\pi\) be a union of \(\pi'\) and the column
\(\{(i_0,j_0,l)|l\ge -k_{i_0,j_0}\}\), respectively \(\mathrm{F}_v(\pi)\) is the generating function of the points inside \(\pi\). The product of the second
and third terms in \ref{eq:ExpDiff} is equal to \(\mathrm{Exp}(-\mathrm{V}_v(\pi)+
\mathrm{V}_v(\pi'))\). On the other hand the difference  \(\mathrm{V}_v(\pi)-
\mathrm{V}_v(\pi')\) is equal
\begin{multline}\label{eq:diffVV}
  \left(\frac{t_3^{-k}-1}{1-t_3}\right)t_1^{i_0}t_2^{j_0}+\left(\frac{t_3^k-1}{1-t_3}\right)
  \frac{t_1^{-i_0}t_2^{-j_0}}{t_1t_2}+\frac{(1-t_1)(1-t_2)}{t_1t_2}\left[ t_1^{i_0}t_2^{j_0}\left(\overline{\mathrm{F}'_e}\frac{1}{1-t_3}
      +t_3^{-k-1}\overline{\mathrm{F}'_v}\right)+\right.\\
 \left. t_1^{-i_0}t_2^{-j_0}\left(\frac{\mathrm{F}'_e}{1-t_3}-t_3^k\mathrm{F}'_v\right)\right].
\end{multline}
where \(\mathrm{F}'_e=\mathrm{F}_e(\pi')\), \(\mathrm{F}'_v=\mathrm{F}_v(\pi')\) and \(k=k_{i_0,j_0}\).

Let us also make an observation that the product \(\mathrm{Exp}(-\mathrm{V}_v({\pi})+\mathrm{V}_v({\pi'}))\) vanishes is \(\pi\) is not an element of
\(\Pi^\PT_\lambda\). Indeed,  \(\pi\in \Pi^\PT_\lambda\) if and only if we have two inequalities satisfied:
\begin{equation}\label{eq:ptt-cond}
-k_{i_0-1,j_0}\ge -k_{i_0,j_0},\quad -k_{i_0,j_0-1}\ge -k_{i_0,j_0}.\end{equation}

The third factor in the product \eqref{eq:ExpDiff} contains terms that might vanish or acquire a pole. These are the terms occur only
when \(ij=(i_0-1,j_0), ij=(i_0,j_0-1), ij=(i_0,j_0-1)\) and if  we collect these terms together then we get
\begin{equation}
  \label{eq:prob-term}
  [0]_{k_{i_0j_0}-k_{i_0-1,j}}[0]_{k_{i_0j_0}-k_{i_0,j_0-1}}/[0]_{k_{i_0j_0}-k_{i_0-1,j_0-1}}.
\end{equation}

We need to check that expression \eqref{eq:prob-term} vanishes when \(\pi\notin \Pi^\PT_\lambda\) and does not develop a pole. Let first look at the case
when the first inequality in \eqref{eq:ptt-cond} fails but the second holds. Since \(\pi'\in \Pi^\PT_{\lambda'}\) we have \(-k_{i_0-1,j_0-1}\ge -k_{i_0,j_0-1}\) and hence
\(k_{i_0j_0}-k_{i_0-1,j_0-1}\le 0\). Thus the first term of \eqref{eq:prob-term} vanishes but the second and the third do not.

The case when the second inequality in \eqref{eq:ptt-cond} fails but the first one holds is analogous. The last case we need to analyze is when
both inequalities in \eqref{eq:ptt-cond} fail. In this case the first two terms of \eqref{eq:prob-term} vanish. Thus the double zero of the numerator
would compensate possible order one pole of the denominator.

The variable \(z_n\) in the last formula is equal to the equivariant weight of the line bundle \(\calL=\calO^{[n]}/\calO^{[n-1]}\) in the proof of \eqref{eq:intSn} and we can use the formula from proposition~\ref{prop:Hilb-int} to finish a proof of the theorem~\ref{thm:mainPT}. 

\subsection{Proof of theorem~\ref{thm:PT-int}}
\label{sec:proof-theorem}
To compute the PT invariants of \(X\) we can use the formula~\eqref{eq:locP1}. Combining this formula with the
main formula of theorem~\ref{thm:mainPT} we conclude that the statement of theorem holds for
\[\mathrm{F}^d(u,v,q,t)=\sum_{\vec{m},\vec{l}}\Omega(z')\Omega(z'')\mathrm{F}^d_{\vec{m},\vec{l}}\,\,u^{\vec{m}}v^{\vec{l}}\]
\[
\mathrm{F}^d(u,v,q,t)=
\sum_{\vec{k'},\vec{k''}}q^{|\vec{k'}|+|\vec{k''}|} \Pi'(\vec{k'},z')\mathrm{E}^d(z',z'')\Pi''(\vec{k''},z'')
\prod_{r=1}^{r'}\sum_{i=1}^n \frac{e^{t_3u_r(z'_i+k'_i)}}{\delta(u_r)}\prod_{r=1}^{r''}\sum_{i=1}^n\frac{e^{t_3v_r(z''_i+k''_i)}}{\delta(v_r)}
\]
where \(z'=(z_1,\dots,z_n)\), \(z''=(z_{n+1},\dots,z_{2n})\); \(\Pi'=\Pi\), \(\Pi''\) is \(\Pi\)
after substituting \(t_i\) with
\(s_1=t_1+d_1t_3,s_2=t_2+d_2t_3,s_3=-t_3\) and

\[\mathrm{E}^d(z',z'')=\sum_{|\lambda|=n}J_\lambda(z') \mathrm{E}^d(\lambda,\lambda) J_\lambda(z''), \]

here  \(J_\lambda(z)\) is the interpolation polynomial \cite{Okounkov98} that expresses the the class of the torus fixed point
\(\lambda\) in terms of the Chern roots \(z_i\) of the tautological bundle \(\calO^{[n]}\).

\begin{remark}
  The edge term \(\mathrm{E}^d(\lambda,\lambda)\) becomes the Euler class of the tangent bundle if \(d=(0,0)\).
  Thus the previous formula simplifies and one can write a formula for the descendent PT invariants of
  \(\CC^2\ti \PP^1\) as an \(n\)-iterated residue (instead of \(2n\)-iterated residue for the general \(d\)), see
  \cite{Oblomkov19}.
\end{remark}

\subsection{Specialization of the equivariant parameters}
\label{sec:spec-equiv-param}
In this section we discuss the properties of the bare vertex in particular we provide a different proof of the 'cancelation of poles' lemma from \cite{PandharipandePixton12} and outline the argument from the last paper. It is shown in \cite{PandharipandePixton12} that the capped
one-legged vertex is a rational function of \(q\).

\begin{theorem} \cite{PandharipandePixton12}
  The any partition \(\eta\) the capped fully-equivariant vertex \(\mathrm{Z}_\PT(\prod_k\ch_{i_k}(0)|\omega)\) is a \(q\)-expansion of some rational
  function.
\end{theorem}

The capped vertex related to the bare vertex by the formula:
\[\mathrm{Z}_{\PT}(\prod_k\ch_{i_k}(0)|\eta)=\sum_\lambda \mathrm{Z}^{\bare}_{\PT}(\prod_k\ch_{i_k}|c_\lambda) \mathrm{W}_{\lambda,\mu}\mathrm{S}^\mu_\eta.\]
where the matrix \(\mathrm{S}^\mu_\eta\) is the rubber solution of the quantum differential equation
\[\mathrm{S}^\mu_\eta=\sum_{d\ge n}q^d\left\langle \mu\left\vert\frac{1}{s_3-\psi_0}\right\vert C_\eta\right\rangle_{d,n},\]
here \(C_\eta\)  are the Nakajima cycles in \(S^{[n]}\).
The matrix \(\mathrm{W}_{\lambda,\mu}\) has integrals \(\int_{C_\lambda}c_\mu\) as
entries.

The argument of \cite{PandharipandePixton12} based on the specialization of the equivariant parameters: \(t_1+t_2=ct_3\), \(c\in\ZZ_+\).
When the parameters are specialized both \(\mathrm{Z}^{\bare}_\PT(\dots)\) and \(\mathrm{S}\) are rational functions of \(q\). In particular,
the following 'cancelation' of poles lemma was proven in \cite{PandharipandePixton12} by a very ingenious method:

\begin{lemma}
  The specialization \(t_1+t_2=ct_3\), \(c\in \ZZ_{\ge 0}\) is well-defined and \( \mathrm{Z}^{\bare}_{\PT}(\prod_k\ch_{i_k}|\omega)\)
  is a Laurent polynomial of  \(q\).
\end{lemma}

We show below the weaker version of the lemma by demonstrating that the specialization
is a rational function of \(q\).
Indeed, the specialization in our terms is the specialization to a generic point
of a line \(a_1+a_1=c\) and we need to study  our residue formula
for the bare vertex.

The function \(F_k(z)\) it specializes to the rational function of \(k\) 
\[
  \prod_{j=0}^{c-1}\left(\frac{z-a_1+j}{z-a_1+k+j}\right)
  \left(\frac{z+a_1-c+j}{z+a_1+k+j}\right)\prod_{j=0}^{2c-1}\left(\frac{z+k-c+j}{z-c+j}\right).\]

To iterated residue in the formula for \(\rzz^\bare_{\PT,\vec{k}}(\prod_k \ch_{i_k}|\omega)\) we need to expand \(\Omega(z)\Pi(\vec{k},z)\) in Laurent
power series of \(z_i^{-1}\). From the formulas above we see that
there are universal polynomials \(Q_{\vec{m}}(\vec{k})\) that give us the expansion
coefficients
\[ [z^{\vec{m}}]\Omega(z)\Pi(\vec{k},z)=Q_{\vec{m}}(\vec{k}).\]
On the other hand  the other factor in the residue formula is
\[ [u^{i_1}_1\dots u^{i_m}_m]\prod_{r=1}^m\sum_{i=1}^n \frac{e^{t_3u_r(z_i+k_i)}}{\delta(u_r)}\]
and it is  polynomial of \(z\) with coefficients that are polynomials of \(\vec{k}\). Since \(\omega\) is also a polynomial
of \(z_i\) we conclude that there is a unversal polynomial
\(P_{\vec{i},\omega}\) such that
\[\rzz^\bare_{\PT,\vec{k}}(\prod_k \ch_{i_k}|\omega)=P_{\vec{i},\omega}(\vec{k}),\]
for all.
Since \(\sum_{\vec{k}\ge 0}P_{\vec{i},\omega}(\vec{k}) q^{|\vec{k}|}\) is a rational
function of \(q\) the statement follows.

\section{Integral formula for the DT invariants.}
\label{sec:integral-formula-dt}

\subsection{DT descendents: definition}
\label{sec:dt-desc-defin}

As we explained in the introduction the DT invariants
\cite{MaulikNekrasovOkounkovPandharipande06a} are the integrals over the virtual class of the Hilbert scheme \(\Hilb(X,\beta)\).
Below we give  precise the definitions of the corresponding cohomology classes and fix the notation for the rest of the paper.

Given a sheaf \(\bss\) over the product \(\Hilb(X,\beta)_n\times X\)  and \(\alpha\in H^*(X)\)
we define  the descendent classes in terms of characteristics classes of \(\bss\):
\[\ch_k(\bss;\alpha):=\int_X\ch_k(\bss)\cdot \alpha,\quad \alpha\in H^*(X).\]

In particular, over the product \(\Hilb(X,\beta)\times X\) there is the universal ideal sheaf \(\bii\) and the quotient vector bundle \(\bqq=\mathbb{O}/\bii\). We abbreviate the notation \(\ch_k(\bqq;\alpha)\) to  \(\ch_k(\alpha)\).


\subsection{Localization formula}
\label{sec:localization-formula}
In this section we are concerned with the DT invariant of the local \(\PP^1\) relative to the infinite point. That is
from here till the end of this section \(X\) is assumed to be  \(\PP^1\ti \CC^2\) and the relative divisor  \(D=\infty\ti \CC^2\) if
we do not explicitly say otherwise.
Here we just remind   that the bare vertex
\[\mathrm{Z}^{\bare}_{\DT}(\prod_i\ch_i|\gamma)^\rtt, \quad \gamma\in H^*_T(\Hilb_n(\CC^2))\] is an element of the ring of formal power-series
\(\CC(t_1,t_2,t_3)[q^{-1}][[q]]\) such that we can evaluate the absolute PT invariants of the local \(X=\calO(d_1)\oplus\calO(d_2)\to\mathbb{P}^1\) by
\[\langle \prod_i\ch_{k_i}([0])\prod_{j}\ch_{l_j}([\infty])\rangle_{n[\mathbb{P}^1]}^{\DT}=\sum_{s}\mathrm{Z}^{\bare}_{\DT}(\prod_i\ch_{k_i}|\gamma_s)
  \mathrm{E}^d(\gamma_s,\gamma_r)
\left.  \mathrm{Z}^{\bare}_{\DT}(\prod_j\ch_{l_j}|\gamma_s)\right|_{t_i=s_i}\]
where \(\gamma_i\) is  basis of \(H^*_{T}(\Hilb_n(\CC^2))\) and \(s_3=-t_3,s_2=t_2+d_2t_3,s_1=t_1+d_1t_3\).

We follow the original paper \cite{MaulikNekrasovOkounkovPandharipande06a} and use the same notations
as in \cite{MaulikNekrasovOkounkovPandharipande06a}. The torus fixed locus  is labeled by the honest 3D partitions with
one infinite leg in direction \(t_3\).
The set of 3D partitions with the shape of the infinite leg \(\lambda\) we denote \(\Pi^{\DT}_\lambda\). Respectively, we have the generating function
for the monomials inside \(\pi\in\Pi^{\DT}_\lambda\):
\[\rqq_v(\pi)=\sum_{ijk\in \pi}t_1^it_2^jt_3^k.\]
If \(\lambda\ne \emptyset\) sum \(\rqq_v(\pi)\) is infinite.
The generating function for the leg of shape \(\lambda\) is the polynomial:
\[\rqq_e(\pi)=\sum_{ij\in\lambda}t_1^it_2^j.\]

In these terms the localization weight of DT virtual cycle  at torus fixed point \(\pi\) is defined in terms of \(\mathrm{V}_v^\DT\):
\begin{equation}\label{eq:V_DT}
\mathrm{V}_v^\DT(\pi)=\rqq_v-\frac{\overline{\rqq}_v}{t_1t_2}+\rqq_v\overline{\rqq}_v\frac{(1-t_1)(1-t_2)(1-t_3)}{t_1t_2t_3}+\frac{\mathrm{F}_e}{1-t_3},\end{equation}
\[\mathrm{F}_e=-\rqq_e-\overline{\rqq}_e\frac{1}{t_1t_2}+\rqq_e\overline{\rqq}_e\frac{(1-t_1)(1-t_2)}{t_1t_2}\]

The  weights of the generating function for the character of \(\ch_k([0])\) at the torus-fixed point \(\pi\) is given by
\[\sum \ch_k(\pi)z^k=(1-e^{zt_1})(1-e^{zt_2})(1-e^{zt_3})\sum_{(ijk)\in\pi} e^{(it_1+jt_2+kt_3)z}.\]

As in PT case we define the bare vertex with descendents by the sum over the torus fixed points:
\[\mathrm{Z}^{\bare}_{\DT}(\prod_i\ch_{s_i}|\lambda)=\sum_{\pi\in\Pi^{DT}_\lambda}\frac{q^{|\hat{\pi}|}}{e_{\pi}}\calE_\pi^{\DT}\prod_i \ch_{s_i}(\pi), \]
where the localization factor for the virtual cycle 
is given by \[\calE_\pi^{\DT}:=\mathrm{Exp}(-\mathrm{V}^\DT_v(\pi)).\]

\subsection{DT/PT formula}
\label{sec:dtpt-formula}

The  statement of the theorem~\ref{thm:DTPTdeg0} follows from a careful comparison of the measures \(\calE^{\DT}_{\pi}\) and
\(\calE^{\DT}_{\pi}\). This  comparison is done in the next proposition.

Let us fix a partition \(\mu\), \(|\mu|\), we also fix notation \(\Pi^{\DT}_{\emptyset,\mu}\) for the
set of 3D Young diagrams \(\pi\) such that \(\pi^{(1)}=\mu\).
The  elements of the sets \(\Pi^{\PT}_\mu\) as well as
the elements of the set \(\Pi^{\DT}_{\emptyset,\mu}\) are labeled by the \(n\)-tuple of integers \(k_{ij}\), \(ij\in \mu\).
In particular, \(\pi(\vec{k})\in \Pi^{\PT}_\mu\) corresponds to the inverted 3D partition consisting of
\((i,j,m)\) satisfying inequality \(m\le k_{ij}\). Respectively, \(\pi(\vec{k})\in \Pi^{\DT}_{\emptyset,\mu}\) is a 3D partition
consisting of \((i,j,m)\) satisfying inequalities \(1\le m\le k_{ij}\). 

The  measures the measures \(\calE^{\DT}_{\pi(\vec{k})}\) and
\(\calE^{\PT}_{\pi(\vec{k})}\) depend analytically on paramaters \(\vec{k}\). Thus we can discuss the relation between these measures
for generic values of \(\vec{k}\).
\begin{prop} For generic values of \(\vec{k}\) we have:
  \begin{multline*}
  \frac{\calE^{\DT}_{\pi(\vec{k})}}{\calE^{\PT}_{\pi(\vec{k})}}=\prod_{(ij)\in\mu}\frac{[\mathbf{c}(ij)/t_3+a_1+a_2+1]_{k_{ij}}^2}{[\mathbf{c}(ij)/t_3]_{k_{ij}}^2}\times
  \\ \prod_{(ij),(lm)\in\mu}\frac{[\mathbf{c}(ij)/t_3-\mathbf{c}(lm)/t_3+1]_{k_{ij}}[\mathbf{c}(ij)/t_3-\mathbf{c}(lm)/t_3+1+a_1+a_2]_{k_{ij}}
    }{[\mathbf{c}(ij)/t_3-\mathbf{c}(lm)/t_3]_{k_{ij}}
    [\mathbf{c}(ij)/t_3-\mathbf{c}(lm)/t_3-a_1-a_2]_{k_{ij}}}\times\\
  \frac{[\mathbf{c}(ij)/t_3-\mathbf{c}(lm)/t_3-a_1]_{k_{ij}}[\mathbf{c}(ij)/t_3-\mathbf{c}(lm)/t_3-a_2]_{k_{ij}}}{[\mathbf{c}(ij)/t_3-\mathbf{c}(lm)/t_3+1+a_1]_{k_{ij}}[\mathbf{c}(ij)/t_3-\mathbf{c}(lm)/t_3+1+a_2]_{k_{ij}}}\end{multline*}
\end{prop}

\begin{proof}
  Let us write \(\mathrm{Q}_v(\pi(\vec{k}))\) as a difference
  \[\mathrm{Q}_v(\pi(\vec{k}))= \frac{\mathrm{Q}_e}{1-t_3}-\mathrm{F}'_v,\quad \mathrm{F}'_v=\sum_{(ij)\in \mu}\frac{t_1^it_2^jt_3^{k_{ij}}}{1-t_3},
  \quad \mathrm{Q}_e=\sum_{(ij)\in\mu}t_1^it_2^j.\]
The ratio of the measures in the statement is equal to the \(\mathrm{Exp}\) applied to the difference
\(\mathrm{V}^{\PT}_v(\pi(\vec{k}))-\mathrm{V}^{\DT}_v(\pi(\vec{k}))\) where \(\mathrm{V}^{\PT}_v\) is given by \eqref{eq:V_PT} and
\(\mathrm{V}^{\DT}_v\) is given by \eqref{eq:V_DT} with
\(\mathrm{F}_v=\mathrm{F}'_v\).
The  difference is equal to
\[2(-\mathrm{Q}_v+\frac{\bar{\mathrm{Q}}_v}{t_1t_2t_3})+
(\mathrm{Q}_e\bar{\mathrm{Q}}_v-t_3\bar{\mathrm{Q}}_e\mathrm{Q}_v)\frac{(1-t_1)(1-t_2)}{t_1t_2t_3}.\]
\end{proof}

Next let us observe that the expression
\[g(\vec{k},w,\vec{c})=
  \frac{(1-e^{wt_1})(1-e^{wt_2})(1-e^{wt_3})}{t_1t_2t_3}\sum_{(ij)\in\mu}
\frac{1-e^{k_{ij}wt_3}}{1-e^{wt_3}}e^{c_{ij}w},\quad c_{ij}=\mathbf{c}(ij)\]
is equal to the generating function of \(\rtt\)-characters of \(\ch_k(1)\) at \(\pi(\vec{k})\)
if
\(\vec{k}\ge 0\) and \(\pi(\vec{k})\in \Pi^{\DT}_\emptyset\).
On the other hand if \(\vec{k}\le 0\) the expression is equal to the generating function of \(T\)-characters
of \(-\ch'_k\) at \(\pi(\vec{k})\) if \(\pi(\vec{k})\in\Pi^{\PT}_\mu\).

Combining the formula for the bare vertex with the previous proposition we get:
\begin{multline*}
  f(\vec{k},\vec{z},\vec{w})=\prod_{i=1}^n\frac{[z_i+1+a_1+a_2]_{k_i}}{[z_i]_{k_i}}
  \prod_{j=1}^m g(\vec{k},w_j,\vec{z})
  \prod_{1\le i<j\le n} F^{-1}_{k_i-k_j}(z_i-z_j)
  \\
  \prod_{1\le i, j\le n} \frac{[z_i-z_j+1]_{k_{i}}[z_i-z_j+1+a_1+a_2]_{k_{i}}
   [z_i-z_j-a_1]_{k_{i}}[z_i-z_j-a_2]_{k_{i}} }{[z_i-z_j]_{k_{i}}
    [z_i-z_j-a_1-a_2]_{k_{i}}[z_i-z_j+1+a_1]_{k_{i}}[z_i-z_j+1+a_2]_{k_{i}}}.
 \end{multline*}

 Now the first formula in theorem \ref{thm:DTPTdeg0} follows from the
 formula for the bare vertex in theorem \ref{thm:mainPT} if we set:
 \[f(\vec{k},\vec{z},\vec{w})_\mu=f(\vec{k},\vec{z},\vec{w})J_\mu(\vec{z}),\]
 where \(J_\mu(\vec{z})\) is an interpolation polynomial \cite{Okounkov98}
 that represents the torus fixed point \(\mu\in \Hilb_n(\CC^2)\) in terms
 of Chern roots of \(\calO^{[n]}\).

 Also the same argument as in \ref{thm:mainPT} implies that for \(\vec{k}\ge 1\)
 the function \(f(\vec{k},\vec{z},\vec{w})\) vanishes if \(\pi(\vec{k})\notin \Pi^{\DT}_\emptyset\). Thus we have shown also the second formula in theorem \ref{thm:DTPTdeg0}.

\section{DT/PT correspondence: explicit conjectures}
\label{sec:dtpt-corr-expl}

\subsection{Rationality conjectures}
Let us first restate the rationality conjectures for PT invariants
\begin{conj}\cite{PandharipandeThomas10} For any classes $\alpha_i\in H^*(X)$ and any collection 
of positive integers $k_i$ the invariant
 $$\rzz^{X}_{\PT,\beta}(\ch_{k_1}(\alpha_1),\dots,\ch_{k_m}(\alpha_m))$$
 is a rational function of $q$ which is regular outside of the roots of $1$.
\end{conj}
There is even more precise conjecture in \cite{PandharipandePixton13} which in particular gives a bound
(in terms of disibility of $\beta$)
on the degree of the roots of $1$ where the singularities of the rational function might occur.

The conjecture is proven in the case when  \(X\) is Calabi-Yau \cite{Toda10},
\cite{Bridgeland11} and in the case of complete intersection in a toric variety \cite{PandharipandePixton14}.
The rationality is also known for the equivariant theory of the toric  varieties \cite{PandharipandePixton14}.

In DT theory it is customary to factor out degree \(0\) invariants and
work with the renormalized generating functions:
\[\rzz^{\prime,X}_{\DT,\beta}(\ch_{k_1}(\alpha_1)\dots \ch_{k_m}(\alpha_m)):=
\rzz^{X}_{\DT,\beta}(\ch_{k_1}(\alpha_1)\dots \ch_{k_m}(\alpha_m))/\rzz^{X}_{\DT,0}(1)\]

As we explained in the introduction in DT theory the rationality property does not hold as one can see
from the dilaton equation. It is expected though that rationality holds in the stationary sector of the theory:

\begin{conj}\cite{OblomkovOkounkovPandharipande18} For any classes $\alpha_i\in H^{>0}(X)$ and any collection 
of positive integers $k_i$ the invariant
 $$\rzz^{\prime,X}_{\DT,\beta}(\ch_{k_1}(\alpha_1),\dots,\ch_{k_m}(\alpha_m))$$
 is a rational function of $q$ which is regular outside of the roots of $1$.
\end{conj}

The conjecture is proven in the case of complete toric varieties in \cite{OblomkovOkounkovPandharipande18}.

\subsection{Non-equivariant DT/PT correspondence with descendents}
DT invariants are not rational function of $q$ but the conjecture below links them 
to PT invariants thus provides a `regularization` of the DT invariants. Let us fix
notations first:
$$ M(q)=\prod_{i=1}^\infty (1-q^i)^{-i},\quad \mathfrak{F}_1=q\frac{d}{dq}\ln M(-q),$$
$$ \mathfrak{F}_1^{(i)}=\left(q\frac{d}{d q}\right)^i \mathfrak{F}_1.$$
Let us also introduce the `Euler class':
$$ \mathrm{e}=c_1c_2-c_3\in H^6(X).$$
For example, in the case $X=\mathbb{P}^3$, we have $\mathrm{e}=20[\pt]$.

\begin{conj} \label{conj:Frank}
For any classes $\alpha_i\in H^{>0}(X)$ and any collection 
of positive integers $k_i$, $l_j$ we have
\begin{multline*} 
 \rzz^{\prime,X}_{\DT,\beta}(\prod_{i=1}^n\ch_{k_i}(\alpha_i)
\prod_{j=1}^m \ch_{l_j}(1))=\rzz^{X}_{\PT,\beta}(\prod_i\ch_{k_i}(\alpha_i)\prod_j \ch_{l_j}(1))+\\
\sum_{s=1}^m \sum_{J_1,\dots,J_s}f(J)
\rzz^{X}_{\PT,\beta}(\prod_{i=1}^n\ch_{k_i}(\alpha_i)\prod_{r=1}^s \ch_{|| J_r ||}(\mathrm{e})
\prod_{j\in\bar{J}}\ch_{l_j}(1)).
\end{multline*}
where the sum is taken over $s$-tuples of the disjoint subsets $J_r\subset\{1,\dots,m\}$;
$\bar{J}=\{1,\dots,m\}\setminus \cup_r J_r$
\begin{gather*}
||J_r||=\sum_{j\in J_r} (l_j-3),\quad f(J)=\prod_{r=1}^s f(J_r),\\
f(J_r)=(-1)^{|J_r|-1}\mathfrak{F}_1^{(|J_r|-1)}  \frac{ ||J_r||!}{\prod_{j\in J_r}(l_j-3)!}.  
\end{gather*}
\end{conj}

 The formula was checked numerically in case of $X=\mathbb{P}^3$ in degree $1$ and $2$.
%
An example for \(X=\mathbb{P}^3\), \(\beta=[L]\):
\begin{multline*}
\rzz^{\prime,X}_{\DT}(\ch_7(1)\ch_6(1))=\rzz^X_{\PT}(\ch_7(1)\ch_6(1))+
\mathfrak{F}_1 \rzz^X_{\PT}(\ch_4(\mathrm{e})\ch_6(1))+ \mathfrak{F}_1 \rzz^X_{\PT}(\ch_7(1)\ch_3(\mathrm{e}))
\\+(\mathfrak{F}_1)^2 \rzz^X_{\PT}(\ch_4(\mathrm{e})\ch_3(\mathrm{e}))
-\binom{7}{4} \mathfrak{F}^{(1)}_1\rzz^X_{\PT}(\ch_7(\mathrm{e})).
\end{multline*}

\subsection{Equivariant wall-crossing formula}
\label{sec:equiv-wall-cross}

In this section we present a conjectural relation between the equivariant
DT and PT theories with descendents. During 2017 Summer REU
project it was discovered experimentally that in the case
of one insertion there is a very simple relation, essentially the
relation is a simple as one can hope for. The most of computational
experiment was carried by James Hagborg and we are preparing a preprint with
proof of the simplest case of the conjecture.

In this note though we just present the conjecture without much of discussion
of the proofs. For the prettiest   form of of the conjecture is we need to use
a slight renormalization of the standard descendents. Below we define the
descendents in terms of the weights of these classes at the torus fixed
point \(\pi\):
\[\widehat{\ch}[z]=\sum_{m=0}^\infty \widehat{\ch}_m(\pi)z^m=
1-\prod_{i=1}(1-e^{t_iz})\sum_{(i,j,k)\in \pi}e^{(it_1+jt_2+kt_3)z}.\]

In the case of DT theory \(\pi\) is the set of monomials in the basis of the space of the local sections of 
\(\calO_X/\calI\) and \(\widehat{\ch}_k(\alpha)\) differs from the standard descendent
\(\ch_k(\alpha)\) by a sign.
 In the of PT theory \(\pi\) is the set of monomials
 in the space of local sections of \(\calF\) in a stable pair
 \([\calO\to\calF]\). In both DT and PT \(\pi\) is  (in general infinite) 3D Young diagram. We denote
 by \(\widehat{\ch}[z]\) the generating function of the corresponding descendents.

 In the previous
 sections we discussed the capped vertex with one leg, for the definition of general triple-legged
 capped vertices we refer to \cite{MaulikOblomkovOkounkovPandharipande11}, \cite{PandharipandePixton14}.
 It is defined in terms of DT/PT theory of \(X=\mathbb{P}^1\times \mathbb{P}^1\times\mathbb{P}^1\)
 relative to the divisor \(D=D_1\cup D_2 \cup D_3\), \(D_1=0\times \mathbb{P}^1\times\mathbb{P}^1\),
 \(D_2=\mathbb{P}^1\times 0\times\mathbb{P}^1\),
\(D_3= \mathbb{P}^1\times \mathbb{P}^1\times 0\).
 
 \begin{conj}[Hagborg-Oblomkov, 2017]
   For any \(\lambda,\mu,\nu\) we have the relation between the capped vertices:
   \[\rzz^{X/D}_{\DT}(\widehat{\ch}[z]|\lambda,\mu,\nu)^\rtt=\rzz^{X/D}_{\PT}(\widehat{\ch}[z]|\lambda,\mu,\nu)^\rtt \rzz^{X/D}_{\DT,0}(\widehat{\ch}[z])^\rtt.
   \]
 \end{conj}

 The coeffients in front of \(z^i\), \(i=0,1,2\) in the last formulas
 follow from the dilaton equation and DT/PT correspondence without the
 descendents \cite{MaulikOblomkovOkounkovPandharipande11}, \cite{MaulikPandharipandeThomas10}. 
 
 Since the capped vertex is a building block of the DT/PT theories, the conjecture
 for the capped vertex implies the conjecture for the absolute DT/PT theories:

 \begin{conj}
   For any toric \(X\) and any \(\gamma\in H^*_T(X)\) we have:
\[\rzz^{X}_{\DT}(\widehat{\ch}[z](\gamma))^\rtt=\sum_k\rzz^{X}_{\PT}(\widehat{\ch}[z](\gamma_k^{(1)}))^\rtt \rzz^{X}_{\DT,0}(\widehat{\ch}[z](\gamma^{(2)}_k))^\rtt,
   \]
   where \(\sum_k \gamma_k^{(1)}\otimes \gamma_k^{(2)}=\Delta_*(\gamma)\) is the
   Kunneth decomposition for the diagonal embedding \(\Delta:X\to X\ti X\).
 \end{conj}

 The last formula admits non-equivariant specialization and it is easy to
 see that this specialization is a particular case of the conjecture~\ref{conj:Frank}.

\end{document}